\documentclass[a4paper,11pt]{article}
\usepackage{lmodern}
\usepackage[T1]{fontenc}
\usepackage{textcomp} 
\usepackage[latin1]{inputenc}
\usepackage{amsfonts,amsmath}
\usepackage{geometry}
\usepackage[english]{babel}
\usepackage{color}
\usepackage[normalem]{ulem}
\usepackage{mathdots}
\usepackage{amssymb}
\usepackage{amsthm}
\usepackage{graphicx}
\usepackage{leftidx}

\title{Why there are no mappings to infinity under the Collatz map (and similar)}
\author{Maarten Wensink \\ maartenwensink@hotmail.com}
\date{\today}
\begin{document}
\theoremstyle{plain}
\newtheorem{theorem}{Theorem}
\newtheorem{corollary}[theorem]{Corollary}
\newtheorem{lemma}[theorem]{Lemma}
\newtheorem{proposition}[theorem]{Proposition}
\theoremstyle{definition}
\newtheorem{definition}[theorem]{Definition}
\newtheorem{example}[theorem]{Example}
\newtheorem{conjecture}[theorem]{Conjecture}
\theoremstyle{remark}
\newtheorem*{remark}{Remark}

\maketitle
\newcommand{\N}{\mathbb{N}}
\newcommand{\0}{\mathbb{N}_0}
\newcommand{\D}{\mathbb{O}}

\tableofcontents

\section{Abstract}
Following up on earlier work, I suggest why there are no mappings to infinity under the Collatz conjecture, nor under other mappings of the generalization $3n+p$, where $p$ is odd.

\section{Introduction}
The Collatz conjecture / 3n+1 problem posits that recursive application of the mapping
\begin{equation}
	C(n):= 	\begin{cases}
					n/2 &\mbox{if n is even}; \\ 
					3n+1 & \mbox{if n is odd}, 
			\end{cases} 
\end{equation}
\newline
on any natural number $n \in \N$, $\N=1,2,3,\ldots$, eventually leads to 1, after which the cycle $\{4,2,1\}$ is repeated indefinitely (Pickover 2009, Lagarias 2010).

In a previous article (Wensink 2018) I set out my reasons to suspect that under a conjugate of the Collatz function, $\N \setminus 1$, $\N=\{1,2,3,\ldots\}$ is partitioned in `strings', subsets of $\N \setminus 1$ that have no cycles, start with an element of $2+3\0$ and end with an element of $3+4\0$, $\0=\{0,1,2,\ldots\}$. Following up on this analysis, I here sketch how the same argument rejects the existence of mappings to infinity (divergences) under the Collatz map. I repeat a part of the earlier analysis for easy reference, taken largely verbatim from Wensink (2018).

\section{The analysis}
The Collatz map is:
\begin{equation}
\label{Collatz}
	C(n):= 	\begin{cases}
					n/2 &\mbox{if n is even}; \\ 
					3n+1 & \mbox{if n is odd},
			\end{cases} 
\end{equation}
\newline
$n \in \N$, $\N=\{1,2,3,\ldots\}$.

Instead we might study the accelerated Collatz map, that sends odd positive integers to odd positive integers:
\begin{equation}
\widetilde{C}(n):=\frac{3n+1}{2^j},
\label{AccelCollatz}
\end{equation}
where $2^j$ is the largest power of 2 that divides $3n+1$, with $n \in \D$, where $\D=\{1,3,5,\ldots\}$, the odd positive integers. $\D$ can be transformed back to $\N$ by the transformation $g: \D \rightarrow \N$ such that 
\begin{equation}
g(n):=\frac{n+1}{2}.
\label{g}
\end{equation}
Thus, the odd positive integers are enumerated: $g(1)=1$, $g(3)=2$, $g(5)=3$, et cetera. For efficient reference and to avoid confusion, I put these enumerated positive integers between brackets: $g(1)=[1]$, $g(3)=[2]$, $g(5)=[3]$, et cetera. When an entire equation is put between brackets, the entire equation refers to the enumerated positive integer space. In addition, I use $x$ to refer to enumerated positive integers, whereas $n$ refers to those positive integers themselves: $g(n)=[x]$. The inverse of $g$ is $g^{-1}: \N \rightarrow \D$ such that
\begin{equation}
g^{-1}([x])=2[x]-1=n.
\label{g-1}
\end{equation}

Conjugating $\widetilde{C}$ through $g$ yields mapping
\begin{equation}
F([x])=g(\widetilde{C}(g^{-1}([x]))),
\end{equation}
which is $F:\N \rightarrow 1+3\0 \cup 3+3\0$, such that
\begin{equation}
\begin{cases}
\label{F}
F(E^n([2+2m]))=[3+3m]|(m,n) \in \0^2, \\
F(E^n([1+4m]))=[1+3m]|(m,n) \in \0^2,
\end{cases}
\end{equation}
with $E: [\N \rightarrow 3+4\0]$ (explanation below) such that
\begin{eqnarray}
\label{E}
E([x]):&=&4[x]-1, \\
\label{E0}
E^0([x]):&=& [x], \\
\label{En}
E^n([x]):&=&E(E^{n-1}([x])).
\end{eqnarray}

\begin{remark}
Notice that the trivial cycle $\{4,2,1\}$ now becomes the trivial loop [1], since 4 and 2 are even and 1 is the first odd natural number.
\end{remark}

It is essential to thoroughly understand $E$. $E$ has the property that
\begin{equation}
F(E^j([x]))=F(E^k([x])) \; \forall (j,k) \in \0^2, \; [x \in \N]
\end{equation}
Thus, $E$ indicates which members of $[\N]$ have the same image as $[x]$ under $\widetilde{C}$ and are therefore `equivalent'; $E$ is not itself the mapping under $\widetilde{C}$. It presents itself as a function, because which elements of $[\N]$ are equivalent to $[x]$ clearly depends on $[x]$. The domain of $E$ is $[\N]$, while its range is $[E(\N)=3+4\0]$. Hence, every $[x \in \N]$ has infinitely many higher equivalents, whereas members of $[3+4\0]$ have at least one lower equivalent. Similarly, $[E(3+4\0)=11+16\0]$ have at least two lower equivalents, and so forth. The phrase ``taking equivalents" means ``applying $E$". Equivalents could be taken of a single $[x]$, forming $[E(x)]$, then $[E(E(x))=E^2(x)]$, and so forth. Equivalents could also be taken of a set. For instance, the set $[2+2\0]$ has equivalents $[E(2+2\0)=7+8\0]$, $[E^2(2+2\0)=E(7+8\0)=27+32\0]$ et cetera, while the set $[1+4\0]$ has equivalents $E[(1+4\0)=3+16\0]$, $[E^2(1+4\0)=E(3+16\0)=11+64\0]$ et cetera.

Indeed, the domain of $F$ as presented in equation (\ref{F}) follows from taking equivalents. It is a partition of $[\N]$ that consists of two collections of equivalent sets. Those that map to $[3+3\0]$ are equivalents of $[2+2\0]$, while those that map to $[1+3\0]$ are equivalents of $[1+4\0]$. The union of all those equivalents of $[2+2\0]$ and $[1+4\0]$, including $[2+2\0]$ and $[1+4\0]$ themselves, is $[\N]$. Thus, $F$ suggests to write $[\N]$ as $[2+2\0 \cup 1+4\0 \cup E(2+2\0) \cup E(1+4\0) \cup E^2(2+2\0) \cup E^2(1+4\0) \cup \ldots = 2+2\0 \cup 1+4\0 \cup 7+8\0 \cup 3+16\0 \cup 27+32\0 \cup 11+64\0 \cup \ldots]$, as this is the way the domain of $F$, which is $[\N]$, presents itself.

We could consider restrictions of $F$ that each pertain to one of these subsets of $[\N]$. This gives $F|_{[2+2\0]}$ being the part of $F$ that maps $[2+2\0]$ (i.e., $[2+2\0]$ is its domain) to $[3+3\0]$ (i.e., $[3+3\0]$ is its range), $F|_{[1+4\0]}$ being the part of $F$ that maps $[1+4\0]$ to $[1+3\0]$, $F|_{E([2+2\0])}=F|_{[7+8\0]}$ being the part of $F$ that maps $[7+8\0]$ to $[3+3\0]$, $F|_{E([1+4\0])}=F|_{[3+16\0]}$ being the part of $F$ that maps $[3+16\0]$ to $[1+3\0]$, et cetera. Conform with the fact that natural numbers are distributed uniformly $(mod \; 2)$, we find that of any and every two consecutive elements of $[\N]$, exactly one is mapped through $F|_{[2+2\0]}$, that of every and any four consecutive elements of $[\N]$, exactly one is mapped through $F|_{[1+4\0]}$, that of every and any eight consecutive elements of $[\N]$, exactly one is mapped through $F|_{[7+8\0]}$, and so forth. The distance in $[\N]$ between any two consecutive elements of $[\N]$ mapped through $F|_{[2+2\0]}$ is $2^1$, the distance in $[\N]$ between any two consecutive elements of $[\N]$ mapped through $F|_{[1+4\0]}$ is $2^2$, while the distance in $[\N]$ between any two consecutive elements of $[\N]$ mapped through $F|_{[7+8\0]}$ is $2^3$, and so forth. These distances, which I call `intervals', are all powers of 2. I use these powers of 2 to refer to the restrictions of $F$ introduced above: $F_1:=F|_{[2+2\0]}$, $F_2:=F|_{[1+4\0]}$, $F_3:=F|_{[7+8\0]}$, and so forth. Just like I use $[x]$ to refer to any element of $[\N]$ without singling out any element of $[\N]$ in particular, so I use $F_z$, $z \in \N$ to refer to any of these restrictions, without specifying which: if I specify $z=1$, I refer to $F_1=F|_{[2+2\0]}$, if I specify $z=2$ I refer to $F_2=F|_{[1+4\0]}$, and so forth. Consecutive elements of $[\N]$ that are mapped through some $F_z$ are found in $[\N]$ at intervals of $2^z$, and of any and every $2^z$ consecutive  elements of $[\N]$, exactly one is mapped through $F_z$. Indeed, this can be generalized to the following definition:
\begin{definition}
\label{Z-prop}
$z$-proportionality: Let the successive restrictions $F_z$ through which some $[x \in \N]$ maps be indexed $i \in 1,\ldots,n$, so that some $[x \in \N]$ is mapped, successively, through $F_{z_1},F_{z_2},\ldots,F_{z_n}$ (in that order). Then this same permutation of mappings through $F_{z_1},F_{z_2},\ldots,F_{z_n}$ occurs at but not before $[x+2^{\sum_{i=1}^n z_i}]$. Indeed, of any and every $2^{\sum_{i=1}^n z_i}$ consecutive elements of $[\N]$, exactly one is mapped through, successively, $F_{z_1}, F_{z_2},\ldots,F_{z_n}$.
\end{definition}

For the purpose of the current paper, it is necessary to take the reverse perspective. Rather than, as above, asking which position a position maps to, we ask by which position(s) a position is mapped to (if any). The reason for this is that if the Collatz conjecture is true, then the Collatz map organizes $\N$ in a graphical tree rooted in [1]. The way to build such a tree would be to start with $[1]$, the root, and then verify which positions map to $[1]$, then which positions map the positions that map to $[1]$, then which positions map to those positions, and so forth. Thus, we need a reverse version of $z$-proportionality, which I call $y$-proportionality and that is developed below. 

Let (recursive) application of $F_l$ (on $[2+3\0]$) be the forward direction, while (recursive) application of $F_l^{-1}$ (on $[3+4\0]$) is referred to as the backward direction. Also, let us give names to the elements of $[3+4\0]$ and $[2+3\0]$. The elements of $[2+3\0]$ are the beginnings of strings in the forward direction (and the end pieces in the backwards direction); let us call these the `tails' of strings. The elements of $[3+4\0]$ are the end pieces of strings in the forward direction (and the starting pieces in the backwards direction); let us call these the `heads' of strings. Such terminology will shorten sentences down the line.

In Wensink (2018), I split up $F$ in its one-to-one part and the part that involves taking equivalents. The one-to-one restriction of $F$ is $F_l: [2+2\0 \cup 1+4\0 \rightarrow 3+3\0 \cup 1+3\0]$, such that
\begin{equation}
\label{Fl}
F_l([2+2m]):=[3+3m]|m \in \0, \; F_l([1+4m]):=[1+3m]|m \in \0,
\end{equation}
where $l$ can be taken to mean ''1 and 2''; it indexes the part of $F$ that refers to positions that do not have a lower equivalent. The reverse of this mapping, which is one-to-one, is $F_l^{-1}: [1+3\0 \cup 3+3\0 \rightarrow 1+4\0 \cup 2+2\0]$, such that
 \begin{equation}
\label{NE}
F_l^{-1}(3m+3)=2m+2|m \in \0, \; F_l^{-1}(3m+1)=4m+1|m \in \0.
\end{equation}

The mapping through $F_l^{-1}$ depends on the residual modulus 3 of a position: a residual of 0 means that a position is mapped through $F_1^{-1}$; a residual of 1 means that a position is mapped through $F_2^{-1}$, while a residual of 2 means that a position is not in the domain of $F_l^{-1}$. Clearly, in $\N$, of every and any three consecutive positions, exactly one has residual 0 (mod 3), exactly one has residual 1 (mod 3), and exactly one has residual 2 (mod 3). This is the proportionality that we are after in reverse direction. Let $y$ denote the residual modulus 3, i.e.
\begin{equation}
\label{y} y(x):= k \in \{ 0,1,2 \} : x-k \equiv 0 \; \text{(mod 3)}.
\end{equation}
Then we can define:
\begin{definition}
$y$-proportionality: of any and every $3$ consecutive elements, exactly one is mapped through $F_1^{-1}$, exactly one is mapped through $F_2^{-1}$ and exactly one is not in the domain of $F_l^{-1}$.
\end{definition}

In the forward direction, we have seen that since $[2 \mapsto 3]$, $[E^{\0}(2)\rightarrow 3]$. In the reverse direction, this would become $[3 \rightarrow E^{\0}(2)]$: a map of one position to infinitely many. Similar goes for all positions. Thus, to have $z$-proportionality in the reverse direction, we need to specify the number of equivalents taken at each point. For instance, in the reverse direction, $[7 \mapsto 9 \mapsto 6]$, the first through $F_2^{-1}$, the second through $F_1^{-1}$. Now at this point we could take, say $[E^5(6)]=5803$ and then continue the application of $F_l^{-1}$: $[F_2^{-1}(5803)=7737]$, $[F_l^{-1}(7737)=5158]$, and so forth. It turns out that the exact same sequence of mappings, i.e. first $F_2^{-1}$, then $F_1^{-1}$, then taking five equivalents, then $F_2^{-1}$ and then $F_1^{-1}$, occurs at $7+3^4\0$, where 4 is the number of applications of $F_l^{-1}$; $y$-proportionality is unaffected by specified patterns of equivalents. To give a more general definition: 
\begin{definition}
\label{y-prop}
$y$-proportionality: In the reverse direction, let the successive restrictions of $F_l^{-1}$ through which some $[x \in \N]$ maps (so either $F_1^{-1}$ or $F_2^{-1}$) be recorded. Let the number of equivalents taken before each application of $F_l^{-1}$ also be recorded. Let this information be stored as $y$ and be indexed $i=1,2,\ldots,n$. Thus, some $[x \in \N]$ is mapped, successively, through $F_{y_1}^{-1},F_{y_2}^{-1},\ldots,F_{y_n}^{-1}$ (in that order). Then this same permutation of mappings through $F_{y_1}^{-1},F_{y_2}^{-1},\ldots,F_{y_n}^{-1}$ occurs at but not before $[x+3^n]$. Indeed, of any and every $3^n$ consecutive elements of $[\N]$, exactly one is mapped through, successively, $F_{y_1}^{-1}, F_{y_2}^{-1},\ldots,F_{y_n}^{-1}$.
\end{definition}

\begin{definition}
\label{YPS}
$y$-proportional subset: a subset of $\N$ that is $y$-proportional as defined in definition \ref{y-prop}.
\end{definition}

As with $z$-proportionality, such $y$-proportionality revolves around the interval between elements of some set. Clearly, if $[x]$ is mapped through $F_1^{-1}$, so are $[x+3]$ and $[x+6]$, or, more generally, $[x+3\0]$ (interval of 3). Similarly, if $[x]$ is mapped through $F_1^{-1}$ and $[F_1^{-1}(x)]$ is mapped through $F_2^{-1}$, then $[x+9\0]$ (interval of 9) is mapped through $F_1^{-1}$ and $F_1^{-1}(x+9\0)$ is mapped through $F_2^{-1}$. In the Appendix I prove the following:
\begin{enumerate}
	\item $[\N]$ is $y$-proportional.
	\item $[3+4\0]$ is $y$-proportional.
	\item $[E^{\0}(x)]$ is $y$-proportional for all $x$.
\end{enumerate}

With these preliminaries set, I now turn to the main matter of this article.

\section{The pigeonhole principle}
In the earlier article I invoked the pigeonhole principle by partitioning $[\N \setminus 1]$ in sections of so-many consecutive elements, of which so-many were hit exactly once. To give a straightforward example, of any and every 3 consecutive elements of $[\N]$, exactly 3 are in $[1+3\0 \cup 2+3\0 \cup 3+3\0]$, so $[1+3\0 \cup 2+3\0 \cup 3+3\0]$ is a partition of $[\N]$. In some cases in Wensink (2018), the application was more complicated than this, but the principle is always the same. The important thing is that the counting be exact: of any and every $\mathcal{N}$ consecutive elements of $[\N]$, exactly $\mathcal{N}$ need to be hit exactly once to prove something a partition of $[\N]$.

We could further develop the pigeonhole principle as used in Wensink (2018) by giving each position its own subset, say $X$. The formulation will then be as follows: of the \emph{first} $X$ elements of $[\N]$, exactly 1 is in subset $X$. For instance, of the first 7 consecutive elements of $[\N]$, exactly one is in $\{[7]\}$; of the first 12 consecutive elements of $[\N]$, exactly one is in $\{[12]\}$; while of the first 33 consecutive elements of $[\N]$, exactly one is in $\{[33]\}$. This is, of course, edging towards the tautological, but there are good reasons for doing so and I make combinations of such subsets below.

\section{Constructing a tree}
As mentioned above, a natural way of constructing the tree for the Collatz conjecture is to start with the root, $[1]$, and work backwards to see which elements of $[\N \setminus 1]$ are included in that tree. Since the subject of the present paper is divergences (mappings to infinity) under the Collatz conjecture, it is possible to take a more elaborate root, say, the first $\mathcal{N}$ elements of $[\N]$, in short $[\mathcal{N}]$. There may be cycles in $[\mathcal{N}]$, but if a tree built on these positions includes all of $[\N \setminus \mathcal{N}]$, there are no divergences under the Collatz conjecture (but $[\mathcal{N}]$ may include cycles).

Working backwards we first take the first higher equivalent of $[\mathcal{N}]$, $[E(\mathcal{N})]$. These are all heads of strings, which means that the entire strings that these positions are heads of are part of the same tree. We include these strings. All elements in all strings have first higher equivalents, so next we include all of those. These again are all heads of strings, which means that those strings are part of the same tree in their entirety, so we include these. And so forth. For instance, take the string $[2,3]$ ([3=E(1)]). Then we would next include the strings of which $[E(2)=7]$ and $[E(3)=11]$ are heads. Since $[7]$ is the head of string $[5,4,6,9,7]$ while $[11]$ is the head of string $[11]$, we include those strings. We next include the strings to which $[E(5)]$, $[E(4)]$, $[E(6)]$,  $[E(9)]$,  $[E(7)]$ and  $[E(11)]$ are heads, and so forth. This process will be referred to as tree building through iterations.

If we pursue tree building through iterations starting with $[\mathcal{N}]$, many of the strings included will contain elements of $[\mathcal{N}]$. This fact gives rise to a risk of double counting (if we count), so we disregard $[\mathcal{N}]$ and now consider all strings that have an element of $[E(\mathcal{N})]$ as head. Next, we take the first higher equivalents of all elements in all those strings and include the strings these are heads of. Since these newly included position are all different from the positions that were already included, we now simply continue this process.

It is now necessary to relate this tree building process to the pigeonhole principle discussed above. When $[\mathcal{N}]$ is taken as a root, within $[\mathcal{N}]$ there is parity between the number of pigeons and the number of pigeonholes: there are exactly $\mathcal{N}$ pigeons in $[\mathcal{N}]$, and exactly $\mathcal{N}$ pigeonholes. It would be simple if now every element of $[3+4\0]$ were the head of a string that exactly included the elements between that head and the head 4 positions lower, for instance if [443] were the head of string $[440,441,442,443]$ ([439] being the head 4 positions lower). This way, all the pigeonholes would remain filled every time we expand the tree. Taking equivalents of $[\mathcal{N}]$ would then mean that the pigeonholes that we consider increased from $[\mathcal{N}]$ to $[4\mathcal{N}]$, but all these pigeonholes would be nicely filled: of the first $4\mathcal{N}$ elements of $[\N]$, exactly $4\mathcal{N}$ would be included in the tree. Repeating such a process indefinitely would then cover $[\N]$ and the proof would be complete.

But the case at hand is not so straightforward. The fact that taking equivalents implies a multiplication by 4 (since $[E(x)=4x-1]$, -1 becoming negligible as $x \rightarrow \infty$), while the strings thusly included consist on the average of 3 elements (the geometric series that follows immediately from $y$-proportionality) presents a potential problem: it suggests that the number of pigeonholes could multiply by 4, but the number of pigeons would only multiply by 3 every time we expand the tree, potentially leaving many (indeed increasingly many) pigeonholes unfilled.

It is important to find the number of pigeonholes a pigeon is restricted to. Three pigeons in three pigeonholes means that all pigeonholes are covered, whereas three pigeons in more than three pigeonholes means that they are not. Hence the reference above to the construct already alluded to above: including position $[x]$ implies that in the first $x$ elements of $[\N]$, exactly one is hit. When growing the tree, we should keep track of both the number of pigeons and the range of $[\N]$ that these pigeons are restricted to, as this determines the number of pigeonholes.

For every element in an included string, in the next iteration an entire string will be included. On the average, such a string consists of three elements: exponential growth (in a fashion) of the number of pigeons. Such a string, however, starts with a head that is asymptotically four times as large as the element it is the first higher equivalent to. This too implies exponential growth (in a sense).  Thus, each iteration of tree building pigeons and pigeonholes grow on average by a factor 3, respectively 4. So the pigeons seem to be losing. Or are they?

It is certain that the moment we take the first higher equivalent of $[\mathcal{N}]$ and include their strings, on average the number of pigeons triples (since the application of $F_l^{-1}$ depends on $y$, or the residual modulus 3). Indeed, as $\mathcal{N} \rightarrow \infty$, this becomes exactly true. But is the number of pigeonholes really multiplied by 4? Somehow, recursive application of $F_l^{-1}$ on $[3+4\0]$ managed to cover all of $[\N \setminus 1]$, even though there are three pigeonholes to be filled between consecutive elements of $[3+4\0]$ while elements of $[3+4\0]$ create two pigeons on the average ($[3+4\0]$ is y-proportional). Hence, I now look into the way elements included in strings are dispersed over $[\N \setminus 1]$.

\section{String dispersion}
There are two mechanisms through which recursive application of $F_l^{-1}$ on $[3+4\0]$ does fill the intervals between $[3+4\0]$ in $[\N \setminus 1]$, even though we suspect it may not. The first is that the consecutive application of an equal number of mappings through $F_1^{-1}$ and $F_2^{-1}$ yields a lower position. $[F_1^{-1}(x)=2x/3]$, while $[F_2^{-1}(x)=4x/3\rfloor]$; consecutive application of $F_1^{-1}$ and $F_2^{-1}$ thus yields approximately (and asymptotically) $[8x/9<x]$, so that a (small) majority of the maps produced by recursive application of $F_l^{-1}$ on $[3+4\0]$ lies below the element of $[3+4\0]$ that the map originates from. This could be called the non-ergodic nature of the iterative mapping.

The second mechanism is through one-to-one correspondence in the sense of Cantor: the same way in which $[3+4\0]$ has the same cardinality as $[\N]$ because they can be put in one-to-one correspondence through $E$, so the images of $[3+4\0]$ can fill $[\N \setminus 1 \setminus 3+4\0]$. This works as follows. Suppose that every element of $[3+4\0]$ produced exactly two images: one that is two thirds as large as the original (as if mapped through $F_1^{-1}$), and one that is four thirds as large as the original (as if mapped through $F_2^{-1}$). Then the interval between consecutive images of the first kind would be $2/3$ times as large as the interval between elements of the original. Thus, while of every and any four consecutive elements of $[\N \setminus 1]$, exactly one is in $[3+4\0]$, we find that of every and any eight consecutive elements, exactly three would be in the image of $[3+4\0]$, since $3/2\cdot1/4=3/8$. The interval between images of the second kind would be $4/3$ times as large as the interval between consecutive originals: of every and any sixteen consecutive elements of $[\N \setminus 1]$, exactly three would be in this image of $[3+4\0]$. Taken together, we would then find that of every and any sixteen consecutive elements of $[\N \setminus 1]$, exactly nine are in the image of $[3+4\0]$. This is already more than we would expect if we merely counted the number of pigeons that originate from the members of $[3+4\0]$ in 16 consecutive elements of $[\N \setminus 1]$, which would be $4\cdot2=8$. Through the described dispersion, of every and any sixteen consecutive elements of $[\N \setminus 1]$, exactly nine would be in its image, more than double the number of elements that are in $[3+4\0]$, even though every element of $[3+4\0]$ creates only two images.

There is no limit to increase in the density of the image through such dispersion. If every element of $[3+4\0]$ created two images, one that is half the original, and one that is one-and-a-half times the original, each original would still have two images, which on average are as large as the original, and yet we would find that of every and any two consecutive elements of $[\N \setminus 1]$, exactly 1 is in the first image, while of any and every six consecutive elements of $[\N \setminus 1]$, exactly 1 is in the second image. This would gives eight images in every and any twelve consecutive elements of $[\N \setminus 1]$. This while twelve consecutive elements of $[\N \setminus 1]$ contain only three elements of $[3+4\0]$, which generate only six images. With even stronger dispersion, even more of $[\N \setminus 1]$ would be filled. Clearly, the dispersion cannot cover more than all of $[\N \setminus 1 \ 3+4\0]$, as images are unique. For instance, image of one third the original and one image five thirds of the original (average $(1/3+5/3)/2=1$) is not an option, since images of the first kind alone would already fill $[\N \setminus 1 \setminus 3+4\0]$.

The dispersion under $F_l^{-1}$ is more complex than the simple examples above, because some elements of $[3+4\0]$ create no image under recursive application of $F_l^{-1}$, while others create many. And yet we do know the result: of any and every four consecutive elements of $[\N \setminus 1]$, exactly three are in the image of $[3+4\0]$ under recursive application of $F_l^{-1}$. This suggests that when following the tree-building algorithm above, the parity between pigeons and pigeonholes may be maintained after all. In fact, we already know that it does because $[E^{\0}(x)]$ has the same dispersion properties as $[3+4\0]$ (both are y-proportional). This means that there are no divergences under the Collatz mapping (and neither under any similar mapping when $3n+1$ is generalized to $3n+p$ where $p$ is either in $5+6\0$ or $1+6\0$ (see Wensink previous paper). Yet we can look into it in more detail.

\section{Pigeon-pigeonhole parity}
\begin{lemma}
Take root $[\mathcal{N}]$, the first $\mathcal{N}=3^k, k=1,2,3,\ldots$ elements of $[\N]$. As $k \rightarrow \infty$, adding the strings of which $[E(\mathcal{N})]$ are heads maintains a ratio 1 between pigeons and pigeonholes.
\end{lemma}
\begin{proof}
In $[\mathcal{N}]$ there is parity between the number of pigeons and the number of pigeonholes: there are $\mathcal{N}$ pigeons in $\mathcal{N}$ pigeonholes. The strings that end on $[E(\mathcal{N})]$ map to $[\mathcal{N}]$. As $k \rightarrow \infty$, application of $F_l^{-1}$ on the first $\mathcal{N}$ elements of $[3+4\0]$ will have the following result.

All the strings will have a head, i.e. an element of $[3+4\0]$, which is asymptotically 4 times as large as the element they are the equivalent of, because $[E(x)=4x-1]$. This means that for each element of $[\mathcal{N}]$, we have exactly one corresponding element in the first $4\mathcal{N}$ elements of $[\N]$, or simply $[4\mathcal{N}]$. Hence, we have $\mathcal{N}$ pigeons in $4\mathcal{N}$ pigeonholes.

Of each three consecutive elements of $[E(\mathcal{N})]$, exactly one is mapped through $F_1^{-1}$, which implies multiplication by 2/3. Thus, for each 3 elements in $[\mathcal{N}]$, we have exactly one corresponding element in the first $(2/3)\cdot4\mathcal{N}=8/3\mathcal{N}$ elements of $[\N]$, or simply $[8/3\mathcal{N}]$. Thus, we have $\mathcal{N}/3$ pigeons in $8\mathcal{N}/3$ pigeonholes.

Of each three consecutive elements of $[E(\mathcal{N})]$, exactly one is mapped through $F_2^{-1}$, which implies multiplication by 4/3. This means that for each three consecutive elements in $[\mathcal{N}]$, we have exactly one corresponding element in the first $(4/3)\cdot4\mathcal{N}=16/3\mathcal{N}$ elements of $[\N]$, or $[16/3\mathcal{N}]$. We thus have $\mathcal{N}/3$ pigeons in $16\mathcal{N}/3$ pigeonholes.

Of each nine consecutive elements of $[E(\mathcal{N})]$, exactly one is mapped through $F_1^{-1}$ twice, which implies multiplication by $(2/3)^2$. This means that for each 9 consecutive elements in $[\mathcal{N}]$ we have exactly one corresponding element in the first $(2/3)^2\cdot4\mathcal{N}=16/9\mathcal{N}$ elements of $[\N]$, or $[16/9\mathcal{N}]$. We thus have $\mathcal{N}/3$ pigeons in $16\mathcal{N}/3$ pigeonholes.

Of each nine consecutive elements of $[E(\mathcal{N})]$, exactly one is mapped through $F_2^{-1}$ twice, which implies multiplication by $(4/3)^2$. This means that for each 9 consecutive elements in $[\mathcal{N}]$ we have exactly one corresponding element in the first $(4/3)^2\cdot4\mathcal{N}=64/9\mathcal{N}$ elements of $[\N]$, or $[64/9\mathcal{N}]$. We thus have $\mathcal{N}/3$ pigeons in $64\mathcal{N}/3$ pigeonholes.

Of each nine consecutive elements of $[E(\mathcal{N})]$, exactly one is mapped first through $F_2^{-1}$ and then through $F_1^{-1}$. Also, of each nine consecutive elements of $[E(\mathcal{N})]$, exactly one is mapped first through $F_1^{-1}$ and then through $F_2^{-1}$. Calculations like those above show that each of these gives rise to $\mathcal{N}/3$ pigeons in $32\mathcal{N}/3$ pigeonholes, so we have $2\mathcal{N}/3$ pigeons in $32\mathcal{N}/3$ pigeonholes.

And so on. The series that forms is composed of, delightfully, the Fibonacci sequence divided by $2^{n+1}$:
\begin{equation}
\label{Parity}
\frac{1\mathcal{N}}{4\mathcal{N}}+\frac{1\mathcal{N}}{8\mathcal{N}}+\frac{2\mathcal{N}}{16\mathcal{N}}+\frac{3\mathcal{N}}{32\mathcal{N}}+\frac{5\mathcal{N}}{64\mathcal{N}}+\ldots=1.
\end{equation}
Hence, the ratio between pigeons and pigeonholes remains 1.
\end{proof}

Expression (\ref{Parity}) is the sum of an infinite number of fractions that have the number of pigeons as numerator and the number of pigeonholes as denominator. The number of pigeons grows in proportion to its current number: each iteration, each pigeon adds a string, which on average consists of three pigeons. Thus, assuming the average (and this average applies as $\mathcal{N} \rightarrow \infty$), for each present pigeon three new pigeons will be added in the next iteration. The same is true for pigeonholes: the range within which a pigeon is known to be, which determines the number of pigeonholes, also grows in proportion to its current size. If we know a pigeon to be restricted to, e.g., the first 2 elements of $[\N]$, its first higher equivalent will be restricted to the first 8 elements of $[\N]$, its second higher equivalent to the first 32 elements of $[\N]$, and so forth. On the other hand, knowing that a pigeon is restricted to, say, the first 60 elements of $[\N]$, its first higher equivalent will only be restricted to the first 240 elements of $[\N]$, and so forth. This, then, is the rationale for searching for pigeon-pigeonhole parity: To gauge a tree's ability to cover $[\N]$, we need to keep track of both the number of pigeons and the range where these can be found, which determines the number of pigeonholes. Importantly, $a$ pigeons in the first $b>a$ elements of $[\N]$ is equally effective as $c\cdot a$ pigeons in the first $c\cdot b$ elements of $[\N]$ in covering $[\N]$ through growing a tree. Hence the sum of ratios in expression (\ref{Parity}).

A second important property of the ratio found in expression (\ref{Parity}) is that many of the pigeons that are counted in the numerators are among $[\mathcal{N}]$. Thus, the relevance of this number is not that we have found $3\mathcal{N}$ \emph{new} pigeons, but that we have managed to let the tree grow while maintaining coverage of pigeonholes. To give an easy example, if $\mathcal{N}=27$, this includes both $[2]$ and $[E(2)=7]$, so $[7]$ in the next iteration would be counted double if we keep the positions included before the iteration.

\section{No divergences (mappings to infinity)}
How could Lemma 1 be used in a proof of no divergence? I suggest the following. The number of pigeons after the first iteration is $\mathcal{N}+2/3\mathcal{N}+4/9\mathcal{N}+\ldots=3\mathcal{N}$ if we let $k \rightarrow \infty$. These $3\mathcal{N}$ pigeons are contained in various subsets of the first so-many elements of $[\N]$, some of which are more restrictive than the first $3\mathcal{N}$ elements of $[\N]$, for instance when they restrict the pigeons in $[3\cdot2\cdot2\mathcal{N}/9]$, while others are less restrictive, for instance when they restrict the pigeons in $[3\cdot4\cdot4\mathcal{N}/9]$. Yet the mass effect is that the ratio of pigeons to pigeonholes in these sets combines to 1. Effectively, we now have $3\mathcal{N}$ pigeons in $3\mathcal{N}$ pigeonholes.

Notice that in the previous section we have not used the information that many of the pigeons are restricted to an even smaller range of pigeon holes than was stated. For instance, it was stated that for each element of $[\mathcal{N}]$ there is exactly one corresponding string head in the first $4\mathcal{N}$ elements of $[\N]$. In reality, of course, the first head is restricted to the first 4 elements of $[\mathbb{N}]$, the second head is restricted to the first 8 elements of $[\mathbb{N}]$, and so on. Not using this information was deliberate, for the following reason.

$[\mathcal{N}]$, $\mathcal{N}=3^k$, is $y$-proportional up to $k$ recursive applications of $F_l^{-1}$: what happens after that, we cannot say, at least not with any regularity and predictability. Since the next step in the tree-building process is to take the first higher equivalent of all the $3\mathcal{N}$ pigeons and include their strings, what can we say about $y$-proportionality of these $3\mathcal{N}$ pigeons? Although the y-proportionality no longer exists because we cannot guarantee the order in which these permutations occur, all possible permutations up to up to $k+1$ now occur. This is guaranteed by the fact that the way these pigeons have been reached is through the complete set of permutations up to $k$. Thus, we have what we might call messy $y$-proportionality: every permutation of $k+1$ mappings through $F_l^{-1}$ occurs exactly once. Because of this messy $y$-proportionality, we cannot use more detailed information as suggested above.

Putting these things together we see that as the tree grows, it does not grow less dense and therefore cannot skip parts of $[\mathbb{N}]$. If we built a tree on a root where the pigeon-pigeonhole ratio is 1:60, and that is somehow $y$-proportional, the algorithm assures that this 1:60 pigeon-pigeonhole ratio is maintained as the tree expands. Like its root, this tree will have major holes in it (59 out of 60 pigeonholes will be empty). Yet if we start with a root where all pigeonholes are filled, we can build a tree on it where all pigeonholes are filled; it is just that the tree consists of many different subsets that add up to pigeon-pigeonhole parity.

Take root $[\mathcal{N}]$, $\mathcal{N}=3^k, k=1,2,3,\ldots, k \rightarrow \infty$ and then, as above, take $E(\mathcal{N})$ and include the strings that belong to these heads. Pigeon-pigeonhole parity has been maintained (Lemma 1) and we now have a messy $y$-proportional set of $3\mathcal{N}$ pigeons. Hence, perform the next iteration, to find a messy $y$-proportional subset of $9\mathcal{N}$ in $9\mathcal{N}$ pigeonholes. As we progress through this algorithm, the following happens. The first iteration, $[1]$ is no longer reached, because it is not in a string. The second iteration, $[2,3]$ is no longer reached, because $[3=E(1)]$. The third iteration, $[11]$ and $[5,4,6,9,7]$ are no longer reached, because $[7=E(2)]$ and $[11=E(3)]$, and so on. Yet we move through $[\N]$ while pigeon-pigeonhole parity is maintained so that after any number of iterations the resulting set of positions is always sufficient to build tree on in which pigeon-pigeonhole parity is, again, maintained. Hence, $[\N \setminus 1 \setminus \mathcal{N}] \mapsto [\mathcal{N}]$. Incidentally, this also means that all potential cycles occur in $[\mathcal{N}]$.

\section{Discussion}
While I realize that this paper is a bit wordy (a better mathematician may be able to formalize the argument more rigorously), I do think that it contains all the basic ingredients necessary for a formal proof. The key point is to demonstrate that the string dispersion behavior gives rise to a tree-growing process that does not leave holes.

As a further point of insight, apparent divergences, i.e. mappings to higher and higher positions for a large number of iterations, are not counter to the conjecture. Rather, these are necessary for the conjecture to be true at all. To see why, consider that in reverse tree building we take equivalents, which multiplies the range that determines the number of pigeonholes by 4, and then apply $F_l^{-1}$, which multiplies the number of pigeons by 3. The only way in which all elements of $[\mathbb{N}]$ can be included in this tree is if at a (much) later iteration, the tree building process reaches into deep pockets of yet unincluded elements of $[\mathbb{N}]$. Looking in the forward direction, which we normally do, this means that mappings are to higher and higher \emph{numbers} for a (very) large number of iterations. Indeed, there is no limit to the number of iterations that map to higher numbers.

\end{document}